%% file: main.tex
\documentclass[letterpaper, 10 pt, conference]{ieeeconf}
\IEEEoverridecommandlockouts
\usepackage[utf8]{inputenc}
\usepackage{amsmath}
\usepackage{amsfonts}
\usepackage{amssymb}
\usepackage{graphicx}
\usepackage{xcolor}
\usepackage{url}
\usepackage{multirow}
\usepackage{float}
\usepackage{subcaption}
\usepackage{wrapfig}
\DeclareMathOperator{\midd}{mid}
\DeclareMathOperator{\diag}{diag}

\newcommand{\rd}{\delta}
\newtheorem{definition}{Definition}
\newtheorem{lemma}{Lemma}
\newtheorem{prop}{Proposition}
\newtheorem{remark}{Remark}
\newtheorem{cor}{Corollary}
\newtheorem{theorem}{Theorem}
\usepackage{etoolbox}
\makeatletter
\patchcmd{\@makecaption}
  {\scshape}
  {}
  {}
  {}
\makeatletter
\patchcmd{\@makecaption}
  {\\}
  {.\ }
  {}
  {}
\makeatother
\def\tablename{Table}

\begin{document}
\title{\LARGE \bf Tight Continuous-Time Reachtubes
for Lagrangian Reachability}
\author{Jacek Cyranka$^{1*}$
Md. Ariful Islam$^{2*}$, 
Scott A. Smolka$^{3}$, 
Sicun Gao$^{1}$, 
Radu Grosu$^{4}$
\thanks{$^*$Jacek Cyranka and Md.\ Ariful Islam contributted equally to this work.
$^{1}$University of California, San Diego,
$^{2}$Texas Tech University, $^{3}$Stony Brook University, $^{4}$Vienna University of Technology.}%
}
\maketitle
\begin{abstract}
We introduce continuous Lagrangian reachability (CLRT), a new algorithm for the computation of a tight and continuous-time reachtube for the solution flows of a nonlinear, time-variant dynamical system. 
CLRT employs finite strain theory to determine the deformation of the solution set from time $t_i$ to time $t_{i+1}$.
We have developed simple explicit analytic formulas for the optimal metric for this deformation; this is superior to prior work, which used semi-definite programming.
%
CLRT also uses infinitesimal strain theory to derive an optimal time increment $h_i$ between $t_i$ and
$t_{i+1}$, nonlinear optimization to minimally bloat (i.e., using a minimal radius) the state set at time $t_i$ such that it includes all the states of the solution flow in the 
interval $[t_i,t_{i+1}]$.
We use $\delta$-satisfiability to ensure the correctness of the bloating.
Our results on a series of benchmarks show that CLRT performs favorably compared to state-of-the-art tools such as CAPD in terms of the continuous reachtube volumes they compute. 
\end{abstract}
\section{Introduction}
Recent work introduced \emph{Lagrangian ReachTube algorithm} (LRT), a new approach for the reachability analysis of continuous, nonlinear, dynamical systems~\cite{Cyranka2017}.
LRT constructs a discrete-time reachtube (or flowpipe) that given a dynamical system, tightly overestimates the set of reachable states at each time point. 

The main idea of LRT was to construct a ball-overestimate in a metric space that minimizes the \emph{Cauchy-Green stretching factor} at every discrete time instant.  LRT was shown to compare favorably to other reachability analysis tools, such as CAPD~\cite{CAPD,ZW1} and Flow*~\cite{CAS,CAS2} in terms of the discrete reachtube volumes they compute on a set of well-known benchmarks. 

This paper proposes a continuous-time-reachtube extension of LRT, the motivation for which is two-fold. First, LRT, while being optimal in the discrete setting, is not sound in the continuous setting: it is not obvious how to find a ball tightly overestimating the dynamics between two discrete points. Second, LRT is not directly applicable to the analysis of hybrid systems, as the dynamics of a hybrid system may change dramatically between two discrete time points due to a mode switch.
%

The main goal of our algorithm, which we call \emph{continuous Lagrangian ReachTube algorithm} (CLRT), is to efficiently construct an ellipsoidal continuous-reachtube overestimate that is tighter than those constructed by available state-of-the-art tools such as CAPD.  CLRT combines a number of techniques to achieve its goal, including \emph{infinitesimal strain theory, analytic formulas for the tightest deformation metric, nonconvex optimization}, and \emph{$\delta$-satisfiability}.  Computing an as tight-as-possible Lagrangian continuous-reachtube overestimate helps avoid false positives when checking if a set of unsafe states can be reached from a set of initial states. 

The class of continuous dynamical systems in which we are interested is described by nonlinear, time-variant, ordinary differential equations (ODEs): 
\begin{subequations}
\label{cauchy}
\begin{align}
  \dot{x}(t)&= f(t, x(t) ),\\
  x(t_0)&= x_0,
\end{align}
\end{subequations}
\noindent{}where $x\colon\mathbb{R}\to\mathbb{R}^n$.  We assume 
$f$ is a smooth function, which guarantees short-term existence of solutions. The class of time-variant systems strictly includes the class of time-invariant systems.

Given an initial time $t_0$, set of initial states $\mathcal{X}\subset\mathbb{R}^n$, and time bound $T\,{>}\,t_0$, CLRT computes a conservative \emph{reachtube} of~\eqref{cauchy}, that is, a sequence of time-stamped sets of states $(R_1,t_1),{\dots},(R_k,t_k=T)$ satisfying:
\[
\text{Reach}\left((t_0,\mathcal{X}\right),[t_{i-1},t_i]) \subset R_i\text{ for }i = 1,\dots,k,
\]
where $\text{Reach}\left((t_0,\mathcal{X}\right),[t_{i-1},t_i])$ denotes the set of all reachable states of ODE system~\eqref{cauchy} in the time interval $[t_{i-1},t_i]$. The time steps are not necessarily uniformly spaced, and are chosen using infinitesimal strain theory (IST).

In contrast to LRT~\cite{Cyranka2017}, which only computes the set of states reachable at discrete and uniformly spaced time steps $t_i$, for $i\,{\in}\,\{1,{\ldots},k\}$, CLRT computes a tight overestimate for the set of states reachable in non-uniformly spaced continuous time intervals $[t_{i-1},t_i]$.  Hence, CLRT computes space-time cylinders overestimating the continuous-time reachtube.

We also note that the LRT approach, as in prior work on reachability~\cite{Fan2015,MA}, employed \emph{semi-definite programming} (SDP) to compute an appropriate weighted norm minimizing the \emph{Cauchy-Green stretching factor} (deformation metric).  We instead derive a very simple analytic formula for the tightest deformation metric. Thus, there is no need to invoke an optimization procedure to find a tight deformation metric, as the formula for the tightest one is now available.  We also provide a very concise proof of this fact.  Moreover, using SDP significantly increases the running time of the algorithm (compared to the approach based on analytical formulas), and can result in numerical instabilities (refer to the discussion in~\cite{Cyranka2017}) and excessive bloating. 

Let us enumerate the key contributions of this work; 1.~Computation of tightest Lagrangian reachtubes using explicit analytic formulas for the deformation metric. 2.~Derivation of continuous-time reachtube bounds by bloating the discrete-time reachtube via nonconvex optimization. 3.~Application of infinitesimal strain theory to adaptive time-step selection. 4.~Computation of the stretching factor is derived in weighted metric spaces, where the natural enclosures are ellipsoids. 5.~Demonstrate improved performance by using ellipsoidal bounds in place of boxes, which are more natural for control theoretic verification problems.

Let $\text{Reach}\left((t_0,\mathcal{X}\right),t_{i-1})\,{\subset}\,B_{M_{i-1}}(x_{i-1},\delta_{i-1})$, where $B_{M_{i-1}}(x_{i-1},\delta_{i-1})$ is the ball computed by LRT for time $t_{i-1}$. To construct a continuous reachtube overestimate for the interval $[t_{i-1},t_i]$, we bloat the radius of this ball to $\Delta_{i-1}\,{>}\,\delta_{i-1}$, until it becomes a tight overestimate for the entire interval; i.e., $\text{Reach}\left((t_0,\mathcal{X}\right),[t_{i-1},t_i])\,{\subset}\, B_{M_{i-1}}(x_{i-1},\Delta_{i-1})$. 

To ensure that the bloating is as tight as possible, we first find the largest time $t_i$ such that the displacement gradient tensor of the solutions originating in $B_{M_{i-1}}(x_{i-1},\delta_{i-1})$ becomes sufficiently close to linear.  Second, to obtain an initial estimate $\hat{\Delta}_{i-1}$, we assume that $f$ in~(\ref{cauchy}) is convex in the interval $[t_{i-1},t_i]$, and solve a convex optimization problem. Third, because for general nonlinear systems $f$ is likely non-convex, we, through an iterative process where $\hat{\Delta}_{i-1}$ is used as an initial value, compute a sound estimate of $\Delta_{i-1}$ using an SMT solver.


We have implemented a prototype of CLRT in C++ and thoroughly investigated its performance on a set of benchmarks, including those used in~\cite{Cyranka2017}.  Our results show that compared to CAPD, CLRT performs favorably in terms of the continuous-reachtube volumes they compute.  Also note that contrary to LRT, CLRT is fully implemented in C++, which significantly improves the runtime performance of our algorithm. Presently, CLRT externally uses CAPD to compute gradients of the flow, but we know how to achieve this independently, and we are currently working on implementing/distributing CLRT as a software library written in C++.

The rest of the paper is organized as follows. Section~\ref{sec:background} provides background on infinitesimal strain theory, LRT, and convex optimization.
Sections~\ref{sec:cctr} and~\ref{sec:os} describe the bloating factor and optimization steps that we use.  Section~\ref{sec:CLRT} presents the CLRT algorithm. Section~\ref{secimplres} contains our experimental results.  Section~\ref{secfuture} offers our concluding remarks and directions for future work.

\section{Background}
\label{sec:background}
This section gives the necessary background, such that paper is self contained. 
We present techniques that are used by the CLRT algorithms presented in Section~\ref{sec:CLRT} to construct overestimating tight continuous reachtubes. 

\subsection{Finite and Infinitesimal Strain Theory}
\label{sec:ist} 
\begin{figure}
  \begin{center}
    \includegraphics[width=0.25\textwidth]{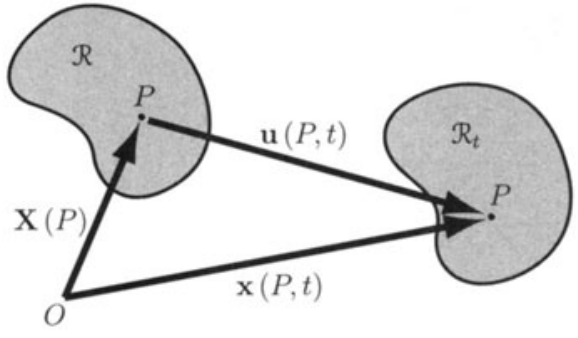}
  \end{center}
  \caption{\scriptsize The reference (or initial) configuration $\mathcal{R}$ and the current configuration $\mathcal{R}_t$ of a body $\mathbb{B}$ subjected to deformation~\cite{Slaughter2002}. A material point $P$ has reference coordinates $\mathbf{X}(P)$ in $\mathcal{R}$, and current coordinates $\mathbf{x}(P,t)$ in $\mathcal{R}_t$, if one uses the same system of coordinates. The displacement vector $u$ shows how the position of a material point $P$ changes from $\mathcal{R}$ to $\mathcal{R}_t$.}
\label{fig:dv}
\end{figure}


A central assumption in continuum mechanics is that a body can be modeled as a continuum, and that the physical quantities distributed over the body can be therefore represented by continuous fields~\cite{Slaughter2002}.

A body $\mathbb{B}$ is composed of an infinite
number of \emph{material points} $P$ and the assignment of each of these material points to a \emph{unique position} in space defines a \emph{configuration} of $\mathbb{B}$. A \emph{reference (or initial)} configuration of $\mathbb{B}$ occupying region $\mathcal{R}$ is used for comparison with the \emph{current} configurations of $\mathbb{B}$ occupying region $\mathcal{R}_t$ at subsequent moments of time $t$. 

Given a material point $P\,{\in}\,\mathbb{B}$, the position vector $\mathbf{X}(P)$ of $P$ relative to a prescribed origin $O$ in $\mathcal{R}$ is called $P$'s \emph{reference position}. The position vector $\mathbf{x}(P,t)$ of $P$ relative to $O$ in $\mathcal{R}_t$ is called the \emph{current position} of $P$. For simplicity, Figure~\ref{fig:dv} uses the same coordinate systems for $\mathcal{R}$ and $\mathcal{R}_t$. However, as we show later, it is convenient to use different coordinate systems or vector bases, which minimize the associated norms of $\mathcal{R}$ and $\mathcal{R}_t$. The coordinates of the reference (undeformed) $\mathcal{R}$ are called \emph{Lagrangian}, whereas the ones of the current (deformed) $\mathcal{R}_t$ are called \emph{Eulerian}.

The \emph{displacement} $\mathbf{u}$ of a material point $P$ from its position in $\mathcal{R}$ to its position in $\mathcal{R}_t$ is defined by the following vector equation:
\begin{equation}
\label{eqn:di}
\mathbf{u}(P,t) = \mathbf{x}(P,t) - \mathbf{X}(P)
\end{equation}
Assuming that each material point $P$ in $\mathbb{B}$ occupies a single position in space at time $t$, there is a (nonlinear) vector operator $\mathbf{\chi}$, mapping $\mathbf{X}(P)$ to $\mathbf{x}(P,t)$, that is, $\mathbf{x}(P,t)\,{=}\,\mathbf{\chi}(\mathbf{X}(P),t)$. Using $\mathbf{\chi}$, the Lagrangian description of the \emph{displacement field} is given by the following equation:
\begin{equation}
\label{eqn:df}
\mathbf{u} = \mathbf{\chi}(\mathbf{X(P)},t)\,{-}\,\mathbf{X}(P)
\end{equation}
A tensor $\mathbf{T}(P,t)$ is a physical quantity associated with the material point $P$ of a body $\mathbb{B}$ at the time $t$. This representation can be given in either Lagrangian coordinates as $\mathbf{T}(P,t)\,{=}\,\mathbf{\Psi}(\mathbf{X}(P),t)$ or Eulerian coordinates as $\mathbf{T}(P,t)\,{=}\,\mathbf{\psi}(\mathbf{x}(P,t),t)$. Since the tensor is the same no matter in which coordinates it is expressed, the Lagrangian description is related to the Eulerian description by:
\begin{equation}
\label{eqn:tle}
\mathbf{\Psi}(\mathbf{X}(P),t)\,{=}\,\mathbf{\psi}(\mathbf{\chi}(\mathbf{X}(P),t),t)
\end{equation}
A particularly important (nonsingular) tensor is the \emph{deformation gradient tensor} $\mathbf{F}$ defined by the following equation:
\begin{equation}
\label{eqn:dgt}
\mathbf{F} = \nabla_X\,\mathbf{x}(P,t) = \nabla_X\,\mathbf{\chi}(\mathbf{X}(P),t)
\end{equation}
By Equations~(\ref{eqn:di},\ref{eqn:df}), the \emph{displacement gradient tensor} is related to the deformation gradient tensor as follows:
\begin{equation}
\label{eqn:dgt}
\nabla_X\,\mathbf{u} = \mathbf{F} - \mathbf{I}
\end{equation}
where $I$ is the identity matrix.
A description of the deformation independent of both translation and rotation is given by a \emph{strain tensor}, of which the \emph{right Cauchy-Green} deformation tensor $\mathbf{C}$ and the \emph{Green-St.~Venant} strain tensor $\mathbf{E}$, are two examples:
\begin{equation}
\label{eqn:st}
\mathbf{C} = \mathbf{F}^{T}\cdot\mathbf{F}\qquad%
\mathbf{E} = (\mathbf{C} - \mathbf{I})/2
\end{equation}
Now by using the definition of $\mathbf{F}$ from Equation~\ref{eqn:dgt}, the Green-St.~Venant strain tensor $\mathbf{E}$ can be rewritten as follows:
\begin{equation}
\label{eqn:gsv}
\mathbf{E} = [(\nabla_X\,\mathbf{u}) +%
              (\nabla_X\,\mathbf{u})^T +%
              (\nabla_X\,\mathbf{u})^T\cdot(\nabla_X\,\mathbf{u})]/2
\end{equation}
If the norm $\|\nabla_X\,\mathbf{u}\|\,{\ll}\,1$, that is, each component of $\nabla_X\,\mathbf{u}$ is of order $O(\varphi)$, for some small parameter $\varphi$, then one speaks about \emph{infinitesimal deformation} and the associated theory is called the \emph{infinitesimal strain theory (IST)}.

In this case, the product $(\nabla_X\,\mathbf{u})^T\cdot(\nabla_X\,\mathbf{u})$ is of order $O(\varphi^2)$, and it can be therefore neglected. This, leads to the linearized version $\mathbf{\varepsilon}$ of $\mathbf{E}$:
\begin{equation}
\label{eqn:gsv}
\mathbf{\varepsilon} = [(\nabla_X\,\mathbf{u}) +%
              (\nabla_X\,\mathbf{u})^T]/2
\end{equation}
which is called the \emph{infinitesimal strain tensor $\mathbf{\varepsilon}$}. Similarly, the \emph{infinitesimal rotation tensor $\mathbf{\omega}$} is defined as follows:
\begin{equation}
\label{eqn:gsv}
\mathbf{\omega} = [(\nabla_X\,\mathbf{u}) -%
              (\nabla_X\,\mathbf{u})^T]/2
\end{equation}
For infinitesimal deformations, $d\mathbf{X}\,{=}\,d\mathbf{x}$, and for any tensor $\mathbf{T}$, the gradients of $\mathbf{T}$ with respect to the Lagrangian and Eulerian coordinates are the same, as $\delta\mathbf{T}/\delta{X}\,{=}\,\delta\mathbf{T}/\delta{x}$~\cite{Slaughter2002}. Hence, in IST it is not necessary to distinguish anymore between Lagrangian and Eulerian coordinates.

\subsection{Review of the LRT Algorithm}
\label{sec:lrt}
The LRT algorithm computes a conservative, discrete-time reachtube for nonlinear, time-variant dynamical systems, based on finite strain theory~\cite{Cyranka2017}. 

The main idea of LRT is to use the right Cauchy-Green strain tensor $\mathbf{C}$ to determine, in a tightest metric, the \emph{stretching factor (SF)} of a ball propagated by the system dynamics in the next time step. 
According to Eqs.~(\ref{eqn:dgt},\ref{eqn:st}), $\mathbf{C}\,{=}\,\mathbf{F}^{T}\cdot\mathbf{F}$,\ \ $\mathbf{F}\,{=}\,\nabla_X\,\mathbf{x}$, and $\mathbf{x}\,{=}\,\mathbf{\chi}(\mathbf{X},t)$, where $\mathbf{\chi}(\mathbf{X},t)$ is the \emph{solution-flow} of the system. $\mathbf{F}$ is the \emph{sensitivity matrix}~\cite{breach}.
\begin{figure}
\vspace{1ex}
\centering
\includegraphics[scale=0.5]{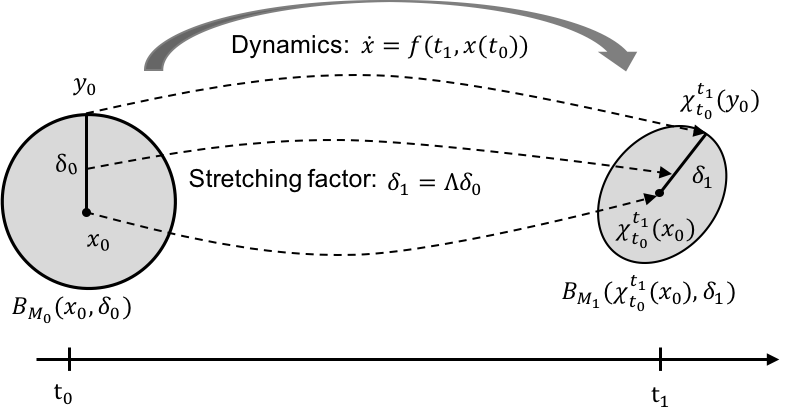}
\caption{Overview of LRT from \cite{Cyranka2017}. Dashed arrows reflect the solution flow $\chi$ and the evolution of state discrepancy.}
\label{fig:intro_fig}
\vspace{-2ex}
\end{figure}

By $\|\cdot\|_2$ we denote the \emph{Euclidean norm}, by $\|\cdot\|_\infty$ we denote the max norm; we use the same notation for the induced operator norms; We use the standard notation $\succ 0$ for positive definiteness. Let $B(x,\rd )$ be the closed ball centered at $x$ with radius $\rd$. $B_M(x,\rd )$ is the closed ball in the metric space defined by matrix $M\succ 0$.
By $\chi_{t_0}^{t_1}$ we denote the flow induced by \eqref{cauchy}.  $\nabla_x\chi_{t_0}^{t_1}$ denotes the partial derivative in $x$ of the flow WRT the initial condition at time $t_1$, which we call the \emph{gradient of the flow}, also referred to as the \emph{sensitivity matrix}~\cite{breach,donze}. 
Let $M\in\mathbb{R}^{n\times n}$ and $M\succ 0$. We say that matrix $M$ defines a metric space when the metric of this space is defined using the  distance function $d_M(x,y) = \sqrt{x^T M y}$.
We denote the norm weighted by $M\succ 0$ as $\|x\|_M = \sqrt{x^T M x}$.

Let  matrices $M_0,M_1\succ 0$ define two metric spaces.
Let $B_{M_0}(x_0,\delta_0)$ be an initial region, given as a ball in metric space $M_0$, centered at $x_0$ and of radius $\delta_0$. Let $y_0$ be a point on the surface of $B_{M_0}(x_0,\delta_0)$, and $x'_0\,{=}\,\mathbf{\chi}_{t_0}^{t_1}(x_0)$ and $y'_0\,{=}\,\mathbf{\chi}_{t_0}^{t_1}(y_0)$, where $\mathbf{\chi}_{t_0}^{t_1}(x)$ abbreviates the solution flow $\mathbf{\chi}(x,t_0,t_1)$ of $x$ when time passes from $t_0$ to $t_1$. Let $\delta_1$ be the distance between $y'_0$ and $x'_0$ in the metric space defined by matrix $M_1$ (see Figure~\ref{fig:intro_fig} for a geometric representation). 

SF $\Lambda$ measures the deformation of the ball
$B_{M_0}(x_0,\delta_0)$ into the ball $B_{M_1}(x'_0,\delta_1)$, i.e. $\Lambda\,{=}\,\delta_1{/}\delta_0$. One can thus use the SF to bound the infinite set of reachable states at time $t_1$ with the ball-overestimate $B_{M_1}(\chi_{t_0}^{t_1}(x_0), \delta_1)$ in an appropriate metric $M_1\succ 0$, which may differ from $M_0\succ 0$. If $M_1\,{=}\,M_0$ we refer to the computed SF as $M_0$-SF or $M_1$-SF, and if $M_0\neq M_1$ we refer to the computed SF as $M_{0,1}$-SF.

LRT's performance on computing tighter overapproximation depends on an appropriate choice of matrix $M_1$. LRT computes $M_1\,{\succ}\,0$ by solving a semi-definite optimization problem.
%
%
Note that the output produced by LRT can be used to compute a validated bound for the \emph{finite-time Lyapunov exponent (FTLE)} $\ln(\Lambda){/}T$, where $T$ is the time horizon. FTLE is used e.g, in climate research to detect Lagrangian-coherent structures~\cite{FTLE}. 
The correctness of the LRT algorithm is given by the following theorem~\cite{Cyranka2017}.  

\begin{theorem}[Thm.~1 in \cite{Cyranka2017}]
\label{thmlrt}
Let $t_0\,{\leq}\,t_1$ be time points, and $\chi_{t_0}^{t_1}(x)$ the solution at $t_1$ of the Cauchy problem~\eqref{cauchy}, with initial condition $(t_0,x)$. Let $M_0,M_1\in\mathbb{R}^{n\times n}$ with $M_0,M_1\,{\succ}\,0$, and $A_0^TA_0 = M_0$, $A_1^TA_1 = M_1$ their respective decompositions.
Let the ball in the $M_0$-norm with center $x_0$ and radius $\delta_0$, $\mathcal{X} = B_{M_0}(x_0,\delta_0)\subseteq\mathbb{R}^n$ be a set of initial states for \eqref{cauchy}. Assume that there exists a compact, conservative enclosure $\mathcal{F}\subseteq\mathbb{R}^{n\times n}$ for the gradients such that:
\begin{equation}
\label{F}
\nabla_x\chi_{t_0}^{t_1}(x)\in \mathcal{F},
\quad \forall x\in\mathcal{X}.
\end{equation}

Suppose $\Lambda\,{>}\,0$ is an upper bound of the whole set of $M_{0,1}$ SFs~\cite{Cyranka2017}, that is:
\begin{equation}
\label{lambdageq}
\Lambda \geq \sqrt{ \lambda_{max}\left( (A_0^T)^{-1}F^TM_1FA^{-1}_0 \right) },
\quad \forall F\in\mathcal{F}.
\end{equation}
where $\lambda_{max}$ represents the maximum eigenvalue.
Then, every solution at time $t_1$ belongs to the
ball:
\begin{equation}
\label{eqn:dtb}
\chi_{t_0}^{t_1}(x)\in B_{M_1}(\chi_{t_0}^{t_1}(x_0), \Lambda \cdot \delta_0).
\end{equation}
\end{theorem}

Observe that ${\sqrt{ \lambda_{max}\left( (A_0^T)^{-1}F^TM_1FA^{-1}_0 \right)}} = {\left\|A_1FA^{-1}_0\right\|_2}$,
where $\|\cdot\|_2$ is spectral (Euclidean) matrix norm.

\section{Explicit Analytic Computation of the Tightest Deformation Metric}
\label{sec:3}
Observe that Thm.~\ref{thmlrt} provides freedom in switching metric spaces used from step to step (by convention, $M_0\succ 0$ denotes the initial metric, and $M_1\succ 0$ denotes the metric of the solution-set bound after one time-step).  To decide if a change of norm should be performed, we compute $\hat{M}_1$ that minimizes the $M_1$-stretching factor $\left({\sqrt{ \lambda_{max}\left( (A_1^T)^{-1}F^TM_1FA^{-1}_1 \right)}}\right)$, and decide based on that; i.e., if the resulting $\hat{M}_1$-stretching factor (SF) is by some measure significantly smaller than the $M_0$-SF $\left({\sqrt{ \lambda_{max}\left( (A_0^T)^{-1}F^TM_0FA^{-1}_0 \right)}}\right)$, then switching to $\hat{M}_1$ may result in a tighter overestimate.

We provide here a surprising argument that the optimal choice of $\hat{M}_1$ minimizing the $M_1$-SF, as well as its decomposition $\hat{M}_1=\hat{A}_1^T\hat{A}_1$ can be obtained using a straightforward computation that we present here. 

Motivated by related work~\cite{Fan2015,MA}, the LRT approach~\cite{Cyranka2017} identified a tight metric by solving a Semi-Definite Programming (SDP) problem.  We show that we do not need to invoke any (convex) optimization technique to find a tight deformation metric, because actually there exist explicit simple analytical formulas for the tightest deformation metric. In particular, this improves upon existing results two-fold: the computation is much faster, and the computed bounds are tighter than the ones computed using SDP. We provide an illustrative example to support our claims in Fig.~\ref{figfirst}. We are convinced that our technique can be applied in the related settings considered in~\cite{Fan2015,MA}, where the authors compute continuous reach-tubes by overestimating solution flows using matrix measures.

Our goal is to minimize the value of $\Lambda$, the upper bound for the $M_{1}$-SF given in Theorem~\ref{thmlrt}.  As finding the best enclosure for a set of SF is a hard problem, we use the following heuristics. The set of gradients $\mathcal{F}\subset\mathbb{R}^{n\times n}$ in our algorithm is given by an interval matrix.  Our choice of $M_1$ is determined by the value of the gradient $F$ being the middle of $\mathcal{F}$, i.e. $F = \midd{(\mathcal{F})}$. We derive an analytical formula for $\hat{M}_1$ minimizing the $M_1$-SF for $F$. We devote the remainder of this section to answer the following crucial question.

\noindent\emph{For a given gradient matrix $F$, what is the $\hat{M}_1$ minimizing the $M_1$ SF?}

\begin{definition}[Analytic $\hat{M}_1\succ 0$]
\label{defchoice}
Let $F\in\mathbb{R}^{n\times n}$ be a full-rank matrix (in our application a gradient of the flow).
Let $V(F)\in\mathbb{C}^{n\times n}$ denote the invertible matrix of \emph{normalized} eigenvectors of $F$ (column-wise). To make this matrix invertible in the case of higher-dimensional eigenspaces (where some eigenvectors are equal), we need to include generalized eigenvectors. For gradients of nonlinear flow equal eigenvectors rarely occurs; hence we do not treat the case of equal eigenvalues in detail.

We define $\hat{M}_1$ as follows:
\begin{equation}
\label{eqAopt}
\hat{A}_1(F) = V(F)^{-1}\quad\text{and}\quad \hat{M}_1(F) = \hat{A}_1(F)^T\hat{A}_1(F)
\end{equation}
When $F$ is known from context, we simply write \\$\hat{A}_1 = \hat{A}_1(F)\text{, and }\hat{M} = \hat{M}_1(F)$.
\end{definition}

We now prove that the choice made in Def.~\ref{defchoice} is optimal, i.e. it minimizes the $M_1$ SF. We remark that our choice of $\hat{M_1}$ is unique by construction using normalized eigenvectors.
\begin{theorem}[$\hat{M}_1$ is optimal]
\label{thmopt}
Let $F\in\mathbb{R}^{n\times n}$ be a full-rank matrix. Let $\hat{A}_1$ and $\hat{M}_1$ be defined by \eqref{eqAopt}. Let the $M_1$-SF be given by
\[
\Lambda(A_1,F) = {\sqrt{ \lambda_{max}\left( (A_1^T)^{-1}F^TM_1FA^{-1}_1 \right)}} = {\left\|A_1FA^{-1}_1\right\|_2}
\]
It holds that
\[
\min_{\substack{A_1\in\mathbb{R}^{n\times n}\\A_1\text{ is invertible}}}{\Lambda(A_1,F)} = \Lambda(\hat{A}_1, F),
\]
i.e., $\hat{M}_1 = \hat{A}_1^T\hat{A}_1$ minimizes the $M_1$-SF.
\end{theorem}
\begin{proof}
First, for arbitrary $A_1$, it holds that
\begin{equation*}
\label{eqsigma}
\sigma_1(A_1FA^{-1}_1) = \|A_1FA^{-1}_1\|_2 = \max_{\|x\|=1,\|y\|=1}{\left|y^TA_1FA^{-1}_1x\right|},
\end{equation*}
where $\sigma_1$ denotes the largest singular value of $A_1FA^{-1}_1$.

Let us pick $y^T = w^T$, and $x = w$, where $w$ is the normalized eigenvector corresponding to the largest eigenvalue of $A_1FA^{-1}_1$.  We have
\begin{multline*}
\|A_1FA^{-1}_1\|_2 = \max_{\|x\|_2=1,\|y\|_2=1}{\left|y^TA_1FA^{-1}_1x\right|} \geq\\\left|w^TA_1FA^{-1}_1w\right| =
|\lambda_{max}(A_1FA^{-1}_1)| = |\lambda_{max}(F)|.
\end{multline*}
Hence, the $M_1$ SF cannot be smaller than $|\lambda_{max}(F)|$. 

Second, we show that this lower bound is in fact attained for $\hat{A}_1=\hat{A}_1(F)$  defined by \eqref{eqAopt}. We have
\begin{multline*}
\|\hat{A}_1F\hat{A}_1^{-1}\|_2 = \sqrt{ \lambda_{max}\left( (\hat{A}_1^T)^{-1}F^T\hat{A}_1^{T}\hat{A}_1F\hat{A}_1^{-1} \right)}\\
=\sqrt{ \lambda_{max}\left( \diag(|\lambda_1|^2,\dots,|\lambda_n|^2) \right) } = |\lambda_{max}(F)|,
\end{multline*}
where $\lambda_1,\dots,\lambda_n$ denote the eigenvalues of $F$.
\end{proof}

\begin{remark}
\begingroup 
\setlength\arraycolsep{1pt}
Let us remark on how we compute matrix $\hat{A}_1$ in practice. Generally, this matrix is complex, as the gradient of the flow is expected to involve some rotation. We do not work within the field $\mathbb{C}$ as we apply algorithms for bounding eigenvalues of real matrices. Instead, we compute the equivalent real matrix $\hat{A}_1$, such that the resulting product $\hat{A}_1F\hat{A}_1^{-1}$ is block-diagonal (having two dimensional blocks corresponding to complex eigenvalues). For example, for $F=\begin{bmatrix}1&1\\-4&1\end{bmatrix}$, we have $\hat{A}_1(F) = \begin{bmatrix}0 & 0.4472\\0.8944  &  0\end{bmatrix}$, and $\hat{A}_1F\hat{A}_1^{-1}=\begin{bmatrix}1&-2\\2&1\end{bmatrix}$.
\endgroup
\end{remark}
We illustrate the Thm.~\ref{thmopt} optimality condition in Figure~\ref{figfirst} (Left).
\begin{figure}[h]
 \includegraphics[scale=0.25]{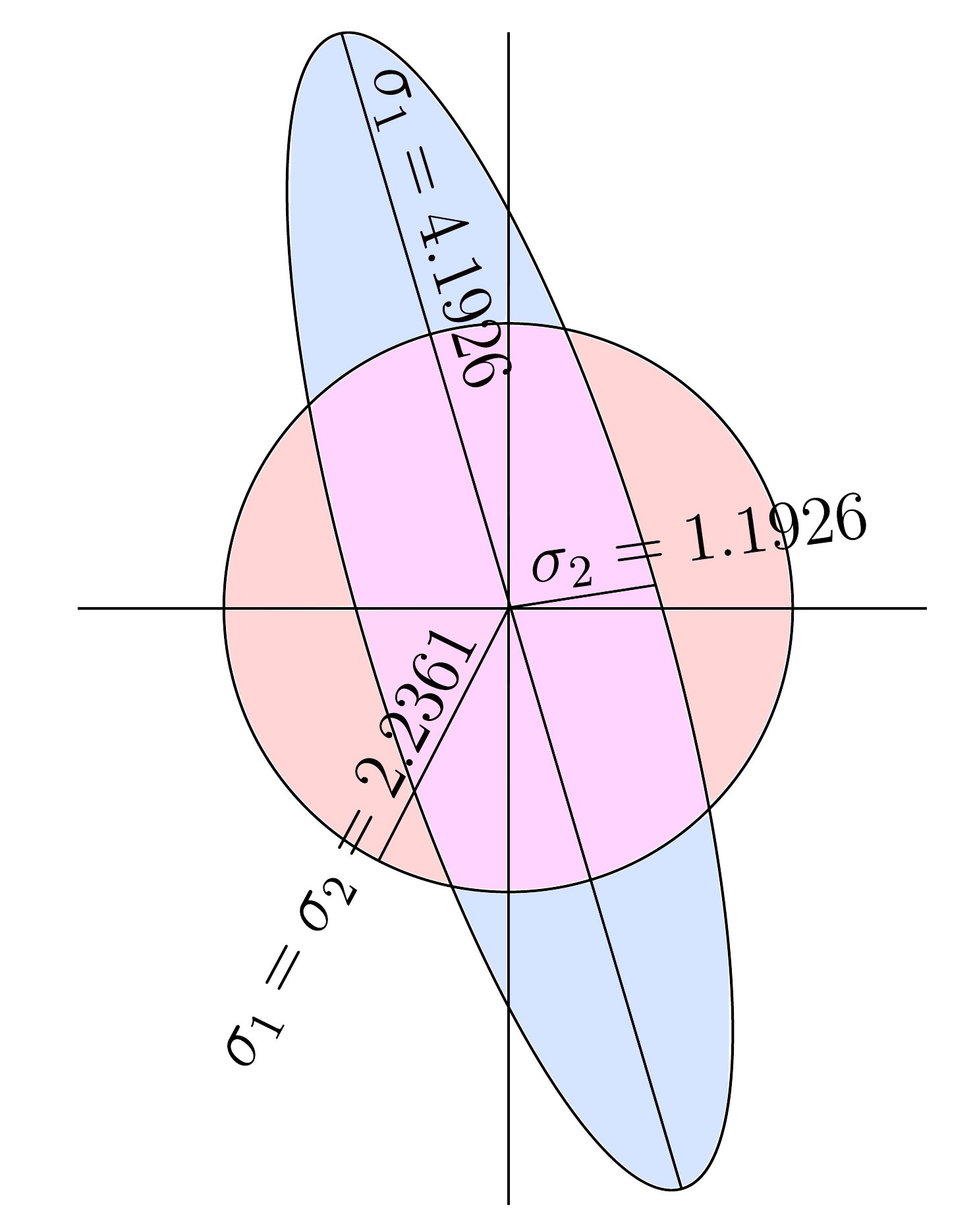}
 \includegraphics[scale=0.45]{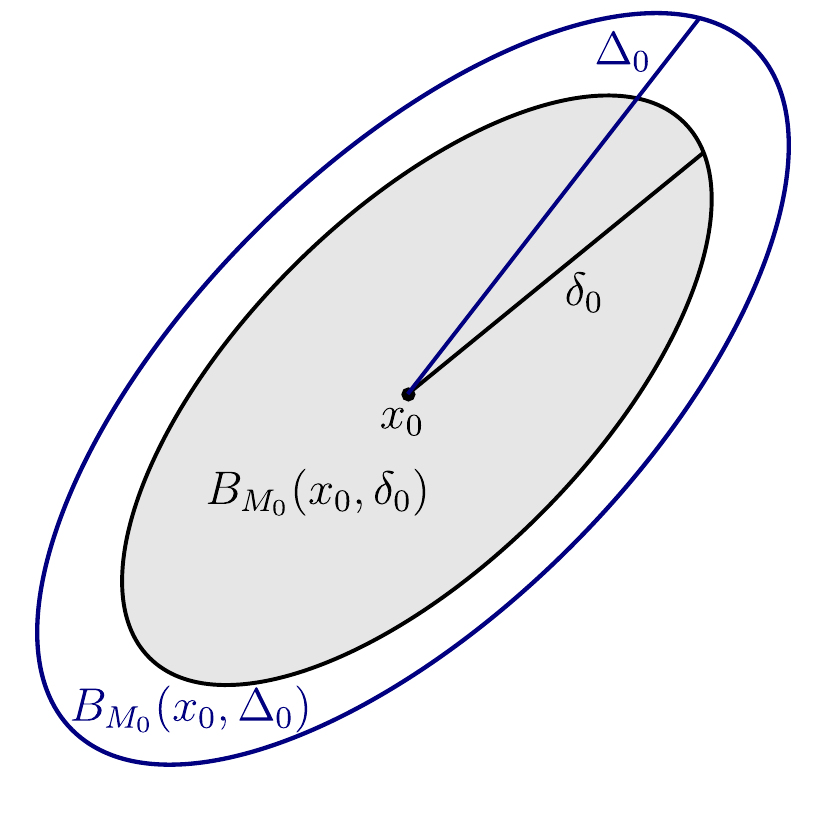}
 \caption[Illustration of optimality condition]{(Left) Illustration of optimality condition of Thm.~\ref{thmopt} for $F = \begin{bmatrix} 1 & 1\\ -4 & 1 \end{bmatrix}$,
 $F$ has conjugate pair of complex eigenvalues $1\pm i\sqrt{2}$. SVD decomposition of $F$ reveals it rotates and transforms unit disc into blue ellipse, where the radii are equal to the singular values ($\sigma_1=4.1926$, $\sigma_2=1.1926$), resp. SVD of $\hat{A}_1F\hat{A}^{-1}_1$ reveals, however, that two singular values are equal $\sigma_1=\sigma_2 = 2.2361$. Recall  SF is equal to $\sigma_1$; although the two ``balls'' have the same volume, the circular one results in a significantly smaller SF ($2.2361$ versus $4.1926$). (Right) The larger ball $B_{M_0}(x_0, \Delta_0)$ depicted in blue, is a conservative over-estimate for the reachtube continuous segment $\text{Reach}\left([t_0,t_0+h],\mathcal{X}\right)$, that is, it is such that $\chi_{t_0}^{[t_0,t_0 + h]}\left(B_{M_0}(x_0,\delta_0)\right) \subset B_{M_0}\left(x_0, \Delta_0\right)$.}
\label{figfirst}
\end{figure}

\section{Conservative Continuous-Time Reachtubes}
\label{sec:cctr}

We present a simple ellipsoidal construction for tightly  over-estimating the continuous-time segments of a reachtube. For a given ellipsoidal bound for the set of initial states of radius $\delta_0$, if the radius of this bound is bloated as in Figure~\ref{figfirst} (Right), up to a computable bound $\Delta_0$, such that:
\[
\delta_0+\max_{\substack{x\in B_{M_0}(x_0,\Delta_0)\\s\in[0,h]}}{\left\|h\cdot f(t_0\,{+}\,s,x)\right\|_{M_0}} \leq \Delta_0,
\]
where $h$ is a time increment and $f$ is the dynamics of the Cauchy problem~\eqref{cauchy}, then the resulting ellipsoid becomes a tight overestimate of the whole continuous segment 
$\chi_{t_0}^{[t_0,t_0+h]}(x)\subset\mathbb{R}^n$, representing the set of all values $\chi_{t_0}^{t}(x)$ of the solution flow of ~\eqref{cauchy}, for all times $t\,{\in}\,[t_0, t_0+h]$. 

\begin{lemma}
\label{lem:lembasic}
Given the Cauchy problem \eqref{cauchy} with $x_0\in\mathbb{R}^n$ the initial state, $t_0$ the current time, $h$ the current time-step, and $\chi_{t_0}^{t}(x)$ the solution flow, let $\Delta\,{>}\,0$ and $M\succ 0$ be a matrix defining the metric space being used. Then:  
\begin{multline*}
\max_{\substack{x\in B_{M}(x_0,\Delta)\\s\in[0,h]}}{\left\|h\cdot f(t_0\,{+}\,s,x)\right\|_{M}} \leq \Delta \\[-2mm]
\Rightarrow\chi_{t_0}^{[t_0, t_0+h]}(x_0) \subseteq B_{M}(x_0,\Delta)
\end{multline*}
\end{lemma}

\begin{proof}
Let $\mathcal{C}([t_0,t_0+h],\mathbb{R}^n)$ denote the space of continuous and differentiable functions defined over the interval $[t_0,t_0+h]$ and domain $\mathbb{R}^n$. Define the operator $T_{x_0}\colon \mathcal{C}([t_0,t_0+h],\mathbb{R}^n)\to\mathcal{C}([t_0,t_0+h],\mathbb{R}^n)$ as follows:
\[
\begin{array}{c}
T_{x_0}(\chi)(h') = x_0 + \int_{0}^{h'}{f(t_0\,{+}\,s, \chi_{t_0}^{t_0\,{+}\,s}(x_0))}\,ds,\ h'\in[0,h].
\end{array}
\]
Let $\mathcal{C}([t_0,t_0+h],B_{M}(x_0,\Delta))\,{\subset}\,\mathcal{C}([t_0,t_0\,{+}\,h],\mathbb{R}^n)$ be the subspace of continuous, differentiable, and bounded functions having their range contained within $B_{M}(x_0,\Delta)$. We show that $T_{x_0}$ maps $\mathcal{C}([t_0,t_0+h],B_{M}(x_0,\Delta))$ into itself. Let  $\chi\in\mathcal{C}([t_0,t_0+h],B_{M}(x_0,\Delta))$. Then $\|T_{x_0}(\chi)(h')-x_0\|_{M}$ is bounded as follows:
\begin{multline*}
\left\|\int_0^{h'}{f(t_0\,{+}\,s,\chi_{t_0}^{t_0\,{+}\,s}(x_0))}\,ds\right\|_M \leq\\[-2mm]
\int_0^{h'}{\left\|f(t_0\,{+}\,s,\chi_{t_0}^{t_0\,{+}\,s}(x_0))\right\|_M}\,ds\leq\\[-2mm]
\max_{\substack{x\in B_{M}(x_0,\Delta)\\s\in[0,h']}}{\left\|h'\cdot f(t_0\,{+}\,s,x)\right\|_{M}}\leq\\[-2mm]
\max_{\substack{x\in B_{M}(x_0,\Delta)\\s\in[0,h]}}{\left\|h\cdot f(t_0\,{+}\,s,x)\right\|_{M}}.
\end{multline*}
The inequalities are due to the fact that $\chi\in\mathcal{C}([t_0,t_0+h],B_{M}(x_0,\Delta))$.
The second and the third ones allow us to compute the integral explicitly. Now using the assumption $\max_{\substack{x\in B_{M}(x_0,\Delta)\\s\in[0,h]}}{\left\|h\cdot f(t_0\,{+}\,s,x)\right\|_{M}} \leq \Delta$,
we can infer that:
\[
T_{x_0}(\chi)(h')\in B_M(x_0,\Delta)\text{ for all }h'\in[0,h].
\]
As $\chi$ was arbitrary, $T_{x_0}\left(\mathcal{C}([t_0,t_0+h],B_{M}(x_0,\Delta))\right)$ is a subset of the class of continuous functions $\mathcal{C}([t_0,t_0+h],B_{M}(x_0,\Delta))$.
For instance, the standard Schauder's fixed-point theorem argument shows that the solution of \eqref{cauchy} with the initial condition $(t_0,x_0)$ -- a fixed point of $T$, satisfies
$
\chi_{t_0}^{t_0+h'}(x_0)\in B_{M}(x_0,\Delta), \forall{h'}\in[0,h].
$
\end{proof}

\begin{theorem}[Optimization for a tight overestimate]
\label{thmcauchy}
Consider the Cauchy problem \eqref{cauchy}, and let $\chi_{t_0}^{t}(x)$ denote the flow generated by~\eqref{cauchy}. Let $x_0\in\mathbb{R}^n$ be an initial state, $t_0$ be the current time, $h$ be the current time-step, and $M_0\succ 0$ be a matrix defining a metric.  Let $B_{M_0}(x_0,\delta_0)$ be given (output from the \emph{LRT algorithm}). Then for all $\bar{x}\in B_{M_0}(x_0,\delta_0)$:
\begin{multline}
\label{eqDelta}
\delta_0+\max_{\substack{x\in B_{M_0}(x_0,\Delta_0)\\s\in[0,h]}}{\left\|h\cdot f(t_0+s,x)\right\|_{M_0}} \leq \Delta_0 \Rightarrow\\[-2mm]
\chi_{t_0}^{[t_0,t_0+h]}(\bar{x})\subset B_{M_0}(x_0,\Delta_0)
\end{multline}

\end{theorem}
\begin{proof}
Pick any $\bar{x}\in B_{M_0}(x_0,\delta_0)$. 
Rewriting the current assumption as
\begin{multline*}
\max_{\substack{x\in B_{M_0}(\bar{x},\Delta_0-\delta_0)\\s\in[0,h]}}{\left\|h\cdot f(t_0+s,x)\right\|_{M_0}}\leq\\
\max_{\substack{x\in B_{M_0}(x_0,\Delta_0)\\s\in[0,h]}}{\left\|h\cdot f(t_0+s,x)\right\|_{M_0}} \leq \Delta_0 - \delta_0,
\end{multline*}
the first inequality holds from $B_{M_0}(\bar{x},\Delta_0-\delta_0)\subset B_{M_0}(x_0,\Delta_0)$.
From Lemma~\ref{lem:lembasic} it immediately follows that
\[
\chi_{t_0}^{[t_0,t_0+h]}(\bar{x})\subset B_{M_0}(\bar{x},\Delta_0 - \delta_0)\subset B_{M_0}(x_0,\Delta_0).
\]
\end{proof}

The following Corollary is obtained by performing minor changes to the proof of Theorem~\ref{thmcauchy}.
\begin{cor}[Applying Thm.~\ref{thmcauchy} backwards in time]
\label{cor:back}
Given the ODE system~\eqref{cauchy}, with $x_0\in\mathbb{R}^n$ the initial state, $t_0$ the current time, $h$ the current time-step, and $\chi_{t_0}^{t}(x)$ the solution flow, let $\Delta_0\,{>}\,0$ and $M_0\succ 0$ define the metric space used. Then for all $\bar{x}\in B_{M_0}(x_0,\delta_0)$:  
\begin{multline}
\label{eqDeltab}
\delta_0 + \max_{\substack{x\in B_{M_0}(x_0,\Delta_0)\\s\in[0,h]}}{\left\|h\cdot f(t_0\,{-}\,s,x)\right\|_{M_0}} \leq \Delta_0 \Rightarrow\\[-2mm]
\chi_{t_0}^{[t_0-h, t_0]}(\bar{x}) \subseteq B_{M_0}(x_0,\Delta_0).
\end{multline}
\end{cor}

\begin{remark}
An important consequence of Theorem~\ref{thmcauchy} and Corollary~\ref{cor:back} is that for \emph{time invariant systems}, conditions~\eqref{eqDelta},\eqref{eqDeltab} imply
\begin{multline}
\delta_0+\max_{x\in B_{M_0}(x_0,\Delta_0)}{\left\|h\cdot f(x)\right\|_{M_0}} \leq \Delta_0\Rightarrow\\
\chi_{t_0}^{[t_0-h,t_0]}(\bar{x})\text{ and } \chi_{t_0}^{[t_0,t_0+h]}(\bar{x})\subset B_{M_0}(x_0,\Delta_0).
\end{multline}
Hence, ball $B_{M_0}(x_0,\Delta_0)$ covers both the forward and backward orbits locally for all times $[t_0-h, t_0+h]$ that initiate within $B_{M_0}(x_0,\delta_0)$.
\end{remark}

\section{Nonconvex Optimization in CLRT}
\label{sec:os}

\subsection{Bounding the Maximal Vector Field in Metric $\mathbf{M}_0$}
\label{smt}

In Theorem~\ref{thmcauchy} of Section~\ref{sec:CLRT}, we give a computable condition for determining a tight overestimate of a continuous reachtube segment, based on the set of initial states in a ball $B_{M_0}(x_0,\delta_0)$, in some metric  $M_0\succ 0$.  More precisely, a conservative radius (denoted by $\Delta_0$) of the continuous reachtube segment overestimate needs to satisfy:
\begin{equation}
\label{eqDeltacond}
\max_{\substack{x\in B_{M_0}(x_0,\Delta_0)\\s\in[0,h]}}{\left\|h\cdot f(t_0+s,x)\right\|_{M_0}} \leq \Delta_0 - \delta_0.
\end{equation}
Verifying conservativeness of $\Delta_0$ requires bounding the maximal vector field value in a metric given by $M_0$, as in the left-hand side of~\eqref{eqDeltacond}. We emphasize that any convex optimization program (COP) for verifying this condition will not be sound, as we neither assume convexity of $f$ in~\eqref{cauchy}, nor does it follow from our approach. 

An interesting problem in this case is to see if $f$ restricted to times from time-step adaptation scheme based on IST is locally convex. If $f$ is still nonconvex, one may need to further decrease the time-step such that $f$ becomes convex in this small range, and a COP can be used to find $\Delta_0$ satisfying \eqref{eqDeltacond}.

Due to a possible lack of convexity, we restate the global optimization problem of bounding the left-hand side of~\eqref{eqDeltacond} as one of $\varepsilon$-satisfiability over the reals~\cite{dreal2,dreal}. (Normally referred to as $\delta$-satisfiability, we use the name $\varepsilon$-satisfiability to avoid confusion with $\delta_0$, denoting a ball radius in our context.)
For an initial guess for $\Delta_0\,{-}\,\delta_0$, given for example by COP, we define the following quantified formula:
\begin{equation}
\label{eqsat}
\exists_{x\in B_{M_0}(x_0,\Delta_0)}\exists_{s\in[0,h]}{\|h\,{\cdot}\,f(t_0+s,x)\|_{M_0} > \Delta_0\,{-}\,\delta_0}.
\end{equation}
We provide below an interpretation of the $\varepsilon$-satisfiability answers for \eqref{eqsat}, where the first answer tells us that $\Delta_0\,{-}\,\delta_0$ is a good bound:
\textsf{UNSAT} means that $\forall x\,{\in}\,B_{M_0}(x_0,\Delta_0)$, $\forall s\,{\in}\,[0,h]$, the inequality $\|h\,{\cdot}\,f(t_0+s,x)\|_{M_0} \le \Delta_0-\delta_0$ holds, and hence \eqref{eqDeltacond} is satisfied. 
%
$\varepsilon$-\textsf{SAT} means that an $\varepsilon$-weakening is satisfiable, i.e., $\exists x$ $\in$ $B_{M_0}(x_0,\Delta_0)$, $\exists s$ $\in$ $[0,h]$,  $\left|\|h\,{\cdot}\,f(t_0+s,x)\|_{M_0}\,{-}\,\varepsilon\right| > \Delta_0\,{-}\,\delta_0$. 

\section{The CLRT Reachability Algorithm}
\label{sec:CLRT}

\noindent\emph{Notation.}
By $[x]$ we denote a product of intervals (a box), i.e., a compact and connected set $[x]\subset\mathbb{R}^n$. We will use the same notation for interval matrices.

\begin{definition}
\label{defconserv}
Given an initial set $\mathcal{X}$, initial time $t_0$, and target time $t_1\geq t_0$, we call the following compact sets:
\begin{itemize}
\item $\mathcal{W}\,{\subset}\,\mathbb{R}^n$ a \emph{tight} reach-set enclosure if 
$\forall{x}\,{\in}\,\mathcal{X}.~\chi_{t_0}^{t_1}(x)\in\mathcal{W}$.
\vspace*{-1mm}\item $\mathcal{F}\,{\subset}\,\mathbb{R}^{n\times n}$ a \emph{conservative} gradient enclosure if $\forall{x}\,{\in}\,\mathcal{X}.~\nabla_x\chi_{t_0}^{t_1}(x)\in\mathcal{F}$.
\end{itemize}
\end{definition}
\newcommand{\reach}{\mbox{Reach}}
\newcommand{\calX}{\mathcal{X}}
Given a set $\mathcal{X}\,{\subset}\,\mathbb{R}^n$ and a time $t_0$, we call a state $x\,{\in}\,\mathbb{R}^n$ \emph{reachable} within time interval $[t_1,t_2]$ if there is an initial state $x_0\,{\in}\,\calX$ at time $t_0$ and a time  $t\in[t_1,t_2]$, such that $x=\chi_{t_0}^t(x_0)$. The set of all reachable states in interval $[t_1,t_2]$ is called the \emph{reach set} and is denoted by $\reach((t_0,\calX),[t_1,t_2])$.

\begin{definition}[\cite{FKXS} Def.~2.4]
\label{defreachtube}
Given an initial set $\mathcal{X}$, initial time $t_0$, and time bound $T$, a \emph{$((t_0,X), T)$-reachtube} of~\eqref{cauchy} is a sequence of time-stamped sets 
$(R_1, t_1),\dots, (R_k, t_k)$ satisfying the following properties:
 (1)~$t_0 \leq t_1 \leq \dots \leq t_k = T$,
 (2)~$\reach((t_0,\mathcal{X}), [t_{i-1}, t_i]) \subset R_i, \forall i = 1,\dots, k$.
\end{definition}
We shall henceforth simply use the name
\emph{reachtube overestimate} of the flow defined by ODE system~\eqref{cauchy}.
We now present the CLRT algorithm for computing tight over-estimations for segments $(R_i,t_i)$, with 
$\reach((t_0,\mathcal{X}), [t_{i-1}, t_i])\,{\subset}\,R_i$,
whose union makes up the complete \emph{$((t_0,X), T)$-reachtube} of~\eqref{cauchy}.
CLRT therefore computes the \emph{whole reachtube overestimate} of the flow defined by \eqref{cauchy}. LRT computes discrete-time slices of the CLRT reachtube.

\label{sec:alg2}
\noindent{\bf Input:}\emph{ODE system~\eqref{cauchy}};
\emph{Parameters:}Time horizon $T$, initial time $t_0$, number of discrete-time steps $k$, and initial time increment $h\,{=}\,T{/}k$ (observe that $h$ may change during execution of the algorithm due to the IST condition);
\emph{Metric}: Positive-definite symmetric matrix $M_0\succ 0$ for initial norm. 
\emph{Initial region}: Bounds $[x_0]\,{\subset}\,\mathbb{R}^n$ for the center, and the radius $\rd_0\,{>}\,0$, for the ball $B_{M_0}(x_0,\delta_0)$ with norm $M_0$ at initial time $t_0$.
\emph{IST threshold}: $\varepsilon_{IST}>0$ -- threshold used to check for smallness of the IST displacement gradient tensor.
\emph{Increment for $\varepsilon$-satisfiability}: $C_\delta>1$ -- increment used for iterative validation of upper bound for maximal speed within bounds using $\varepsilon$-satisfiability.
\emph{Norm switch threshold} $C_M>0$ -- threshold value used to decide if the metric space used should be updated to a new $\hat{M}_1$.
%

\noindent{\bf Output:}
%
$\{ [x_j] \}_{j=1}^k\,{\subset}\,\mathbb{R}^{n\times{k}}$: Interval enclosures for ball centers $x_j$ at time 
$t_0\,{+}\,jh$.
$\{ M_j \}_{j=1}^k$: Norms defining metric spaces for the ball enclosures.
$\{ \Delta_j \}_{j=1}^k\,{\in}\,\mathbb{R}_+^{k}$: Radii of the ball enclosures at $x_j$, for $j = 1,\dots,k$.\footnote{Observe that the radius is valid for the $M_j$ norm, $B_{M_j}([x_j],\Delta_j)\subset\mathbb{R}^n$ for $j = 1,\dots,k$ is a conservative output, that is, $B_{M_j}([x_j],\Delta_j)$ is an over-approximation for the set of states reachable at times  $[t_0,t_1]$ starting from any state $(t_0,x)$, such that $\forall{x}\in \mathcal{X}$:
$
\reach((t_0,\mathcal{X}), [t_j, t_{j+1}]) \subset 
B_{M_j}([x_j], \Delta_j)\text{, for } j = 1,\dots,k.
$}\\[1mm]
%
%
\noindent{\bf Begin CLRT}\footnote{For notational brevity, we use $0$ and $1$ in the subscript to denote $j$ and $(j+1)$, respectively.}
\begin{enumerate}
\item\emph{Begin IST}
\begin{enumerate}
\item Compute overestimates for the (deformation) gradient tensor $[\nabla_x\chi_{t_0}^{t_1}\left([B([x_0],\delta_0\right)])]$, and for the displacement gradient tensor $[\nabla_X{u(X,t)}]$.
\item Adjust the time increment $h$ by halving it until
$
\|[\nabla_X{u(X,t)}]\| < \varepsilon_{IST}\text{ is satisfied.}
$
\item Set $t_1\,{=}\,t_0\,{+}\,h$.
\end{enumerate}
\item\emph{End IST, Begin improved LRT (see section~\ref{sec:3})} 
\begin{enumerate}
\item Compute an enclosure for $[x_1]$, i.e. the center of the reachtube at $t_1$, and for the gradient of the flow initiating at $[x_0]$, i.e. $[D_x\chi_{t_0}^{t_1}([x_0])]$.

\item Compute the optimal deformation metric $\hat{M}_1(F)=\hat{A}_1(F)^T\hat{A}_1(F)$ (see Def.~\ref{defchoice}) for $F=\midd{[D_x\chi_{t_0}^{t_1}([x_0])]}$.

\item If it holds that $M_0\text{-SF} > C_M\cdot \hat{M}_1\text{-SF}$, set $M_1 = \hat{M}_1$. Otherwise, set $M_1=M_0$.

\item Compute an upper bound for $M_{0,1}$-SF \eqref{lambdageq} (denoted $\Lambda$), and compute the discrete-time reachtube overestimate at time $t_1$:
\[
B_1 = B_{M_1}([x_1], \Lambda\cdot\delta_0), \quad
\reach((t_0,\mathcal{X}), t_1) \subset B_1.
\]
\end{enumerate}
\item\emph{End improved LRT, Begin continuous part}
\begin{enumerate}
\item Initialize $\Delta_0 = \delta_0 \cdot C_\delta$.
\item Solve a nonlinear convex optimization problem to compute $\tilde{\delta}$, an approximate maximum of the left-hand side of~(\ref{eqDeltacond}), and set $\hat{\Delta}_0=\delta_0+\tilde{\delta}$. 
\item Update $\tilde{\delta}$ to satisfy an upper bound for the global maximum of the left-hand side of~(\ref{eqDeltacond}) as follows:
\begin{enumerate}
\item check following SMT formula using dReal:
\begin{equation*}
\begin{split}   
\exists_{x\in B_{M_0}(x_0,\hat{\Delta}_0)}\exists_{s\in[0,h]}\|h\,{\cdot}\,f(t_0\,{+}\,s,x)\|_{M_0}\\
>\hat{\Delta}_0\,{-}\,\delta_0 = \tilde{\delta}.
\end{split}
\end{equation*}
\item If dReal returns \textsf{UNSAT}, then $\tilde{\delta}$ is an upper bound for the global maximum. Otherwise, set $\tilde{\delta}=\tilde{\delta}\cdot C_\delta$, and $\hat{\Delta}_0=\delta_0+\tilde{\delta}$, go to step i.
\end{enumerate}	

\item If we set $\Delta_0 = \hat{\Delta}_0$, then (\ref{eqDeltacond}) holds, and thus $\Delta_0$ is an appropriate radius for the  continuous tube. Otherwise, we set $\Delta_0=\Delta_0\cdot C_\delta$ and go to step (b).

\end{enumerate}
\item\emph{End continuous part}
\item For next interval, reset the initial time to $t_1$, and consider the enclosure for the new initial set 
as $B_{M_1}([x_{1}], \rd_1)$. 
\item Save $B_{M_0}([x_0], \Delta_0)$ satisfying
{\small\begin{multline*}
\reach((t_0, B_{M_0}([x_0], \delta_0)), [t_{0}, t_1]) \subset B_{M_0}([x_0], \Delta_0).
\end{multline*}}
\item If $t_1 \geq T$ terminate. Otherwise, go back to 1. 
\end{enumerate}
{\bf End CLRT}
\begin{figure}
    \begin{subfigure}[b]{0.24\textwidth}
        \includegraphics[width=\textwidth]{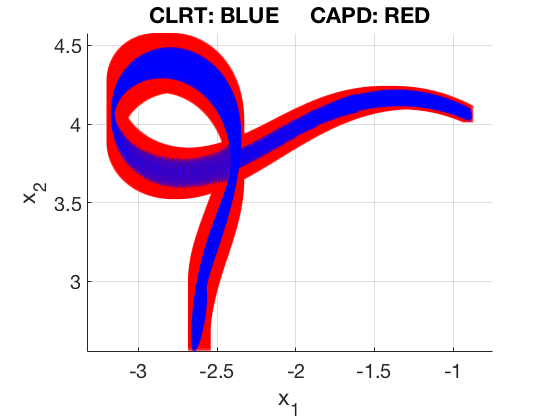}
    \end{subfigure}
 \begin{subfigure}[b]{0.24\textwidth}
        \includegraphics[width=\textwidth]{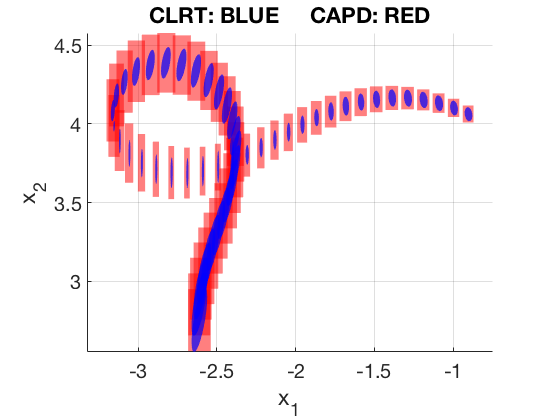}
\end{subfigure}
\caption[Dubins car]{Comparison of continuous-time reachtubes for Dubins car example with nonlinear steering function $\dot{x}=\cos{\theta}, \dot{y}=\sin{\theta}, \dot{\theta}=x\sin{t}$. See also Table~\ref{table:res}, row D3.  Projection of computed bounds onto $x,y$ is shown. The left figure presents all tube segments for times within $[5, 10]$, whereas the right figure shows one segment per $20$ sequential segments computed by the algorithms. Set of initial states
has center at $(0,0,0.7854)$ and radius $0.01$.}
\label{figdubins}
\end{figure}
\begin{prop}
Assume that the rigorous tool used by LRT produces conservative gradient enclosures for~\eqref{cauchy}, and that LRT terminates on the provided inputs. Let $[t_0, T]$ be the whole time interval, which is traversed by CLRT in $k$ steps. Then, the output of the CLRT is a \emph{tight reachtube over-approximation of~\eqref{cauchy}} for all times in $[t_0,T]$; i.e., for $t_{k+1} = T$:
$\reach((t_0,\mathcal{X}), [t_j, t_{j+1}]) \subset 
B_{M_j}([x_j], \Delta_j)\text{, for } j = 1,\dots,k$
\end{prop}
\begin{proof}
The proof follows from \cite[Theorem LRT-Conservativity]{Cyranka2017}, Theorems~\ref{thmopt} and~\ref{thmcauchy}, and the soundness of the $\varepsilon$-satisfiability algorithm. 
\end{proof}

CLRT is also an efficient algorithm.  This follows from the use of IST to derive the proper time increments $h$, and the use of nonlinear COP to finding initial estimates $\hat{\Delta}_0$, which are passed to the $\varepsilon$-satisfiability algorithm.
\vspace{-2ex}
\section{Implementation and Experimental Results}
\label{secimplres}
We implemented a prototype of CLRT in C++. Our implementation is based on interval arithmetic; i.e., all variables used in the algorithm are over intervals, and all computations performed are executed using interval arithmetic. The prototype runs the CLRT algorithm in two passes. 
In the first pass, overestimates for discrete segments of reachtube are generated using the LRT algorithm. In the second pass, we run the procedure for constructing continuous tight overestimates, in each time interval from the discrete ones. In the first pass, we also compute optimal norms by using the analytical formulas from Def.~\ref{defchoice}. 

To compute an upper bound $\Lambda$ for the square-root of the maximal eigenvalue of all symmetric matrices in some interval bounds, we implemented in C++ several algorithms~\cite{hladik,rump,rohn} and used the tightest result available.  Source code, numerical data, and readme file describing compilation procedure for LRT are available online~\cite{codes}.

Table~\ref{table:res} summarizes CLRT's performance on a set of benchmarks (see~\cite{Cyranka2017} for details). Fig.~\ref{figdubins} presents a visual comparison of computed bounds for benchmark D(3). The table illustrates that CLRT performs much better in all cases except Mitchell Schaeffer model and biology model. Typically our algorithm works better for stable systems, whereas CAPD is specialized for chaotic system or a system containing unstable regime in its dynamics.

\section{Related Work}
\label{sec:rw}
Computing the exact reachable set of~(\ref{cauchy}) is hard, as these systems do not admit a closed-form solution. Instead, a conservative (over-approximating) reachtube is computed to determine if an unsafe region can possibly be reached. Existing tools and techniques for conservative reachtube computation
can be classified into three categories according to the time-space approximation they perform : (1)~Taylor-expansion in time, variational-expansion in space (wrapping-effect reduction) of the solution set, e.g., CAPD~\cite{CAPD,ZW1}, VNode-LP~\cite{nedialkov2006vnode,vnode2}, CORA~\cite{Althoff2015a}. (2)~Taylor-expansion in time and space of the solution set, e.g., Cosy Infinity~\cite{makino2006cosy,BM,MB}, Flow*~\cite{CAS,CAS2}.
(3)~Bloating-factor-based and discrepancy-function-based~\cite{FKXS,Fan2015}. Other approaches to compute conservative reachtubes include the Hamilton-Jacobi-based method~\cite{mitchell2003overapproximating,bansal2017hamilton} and the recently developed Runge-Kutta-based method~\cite{alexandreditsandretto}. In this paper, we present an alternative (and orthogonal) technique based on a stretching factor that is derived from an over-approximation of the gradient of the solution-flows (also known as
the sensitivity matrix) and the deformation tensor. 
%
\section{Conclusions and Future Work}
\label{secfuture} 
We presented CLRT, a new algorithm for computing tight reachtubes for solution flows of nonlinear systems. CLRT synergistically combines a number of techniques, e.g.\ finite and infinitesimal strain theory, computation of tightest deformation metric using explicit analytical formulas, $\delta$-satisfiability, and nonconvex optimization. 

Future work includes distributing a C++ implementation of
CLRT, and extending our approach to Hybrid dynamical systems and PDEs. The implementation of our tool will significantly improve its performance on large-scale nonlinear and continuous dynamical systems.

\input{benchmark_table}

\newcommand{\contractnumber}{FA9550-15-C-0030}
\newcommand{\contractingagency}{U.S. Air Force (AFRL)}
\newcommand{\fundingack}{Research
supported by: \contractingagency\ Contract No.~\contractnumber; NSF CPS-1446832; NSF IIS-1447549; NSF CNS-1445770; DARPA Assured Autonomy Program.}
\newcommand{\disclaimer}{Any opinions, findings and conclusions, or 
recommendations expressed
in this material are those of the authors and do not 
reflect the views of the \contractingagency.}

\vspace{2ex}
\subsection{Acknowledgements}
We thank the anonymous reviewers for their valuable comments.
\fundingack
\vspace{-2ex}
\bibliography{cLRT}
\bibliographystyle{abbrv}

\end{document}

%% file: benchmark_table.tex
\vspace{-2ex}
\begin{table}
\tiny
\hspace*{-2ex}
\caption{Performance comparison CAPD. We use following labels: B(2)- Brusselator, I(2)- Inverse Van der Pol oscillator, D(3)- Dubins Car, F(2)- Forced Van der Pol oscillator, M(2)- Mitchell Schaeffer cardiac-cell model, R(4)- Robot arm, O(7)- Biology model, P(12)-Polynomial system (number in parenthese denotes dimension). T: time horizon, dt: time step, ID: initial diameter in each dimension, TV: total volume of reachtubes in T, AV: average volume of reachtubes in T.}
\begin{center}
\bgroup
\def\arraystretch{1.5}%
\begin{tabular}{|c|c|c|c|c|c|c|c|}
\hline 
\multirow{2}{*}{BM} & {$T$} & {ID}	&	
      \multicolumn{2}{c|}{CLRT}&
      \multicolumn{2}{c|}{CAPD}
      \\
      \cline{4-7}
     &&& TV &  AV & TV & AV \\
	\hline 
B(2) &$10$ &$0.02$
&$\mathbf{0.21}$ &$\mathbf{2\times 10^{-4}}$
&$0.59$ &$6\times^{-4}$
\\
\hline
I(2) &$10$ &$0.02$
&$\mathbf{0.13}$ &$\mathbf{6\times10^{-5}}$
&$0.15$ &$7\times10^{-5}$
\\
\hline
D(3) &$10$ &$0.02$
&$\mathbf{0.24}$  &$\mathbf{1\times10^{-4}}$
&$2.8$ &$1\times10^{-3}$
\\
\hline
M(2) &$10$ &$0.002$
&$0.005$ &$2\times10^{-5}$
&$\mathbf{0.003}$ &$\mathbf{2\times10^{-6}}$
\\
\hline
R(4) &$10$ &$0.02$
&$\mathbf{1.2}$ &$\mathbf{1\times10^{-12}}$
&$6.2$ &$1\times10^{-10}$
\\
\hline
O(7) &$5$ &$10^{-4}$ 
&$7\times10^{-16}$ &$7\times10^{-19}$
&$\mathbf{6\times10^{-20}}$ &$\mathbf{6\times10^{-23}}$
\\
\hline
P(12)  &$0.5$ &$10^{-4}$
&$\mathbf{10^{-34}}$ &$\mathbf{10^{-36}}$
&$6\times10^{26}$ &$6\times10^{-26}$
\\
\hline
\end{tabular}
\egroup
\vspace{1ex}
\vspace*{1ex}
\end{center}
\label{table:res}
\vspace{-6ex}
\end{table}